\documentclass[12pt]{article}
\usepackage{amsmath}
\usepackage{amssymb}
\usepackage{amstext}
\usepackage{amscd}
\usepackage{hyperref}

\newtheorem{teo}{Theorem}[section]
\newtheorem{prop}[teo]{Proposition}
\newtheorem{lem}[teo]{Lemma}
\newtheorem{cor}[teo]{Corollary}

\newtheorem{defini}[teo]{Definition}

\newtheorem{rem}[teo]{Remark}
\newtheorem{con}[teo]{Condition}
\newtheorem{conv}[teo]{Convention}

\newcommand{\GL}{{\rm GL}}
\newcommand{\SL}{{\rm SL}}
\newcommand{\Sh}{{\rm Sh}}

\newcommand{\Th}{{\rm Th}}

\newcommand{\Aut}{{\rm Aut}}
\newcommand{\Res}{{\rm Res}}
\newcommand{\MT}{{\rm MT}}

\newcommand{\ab}{{\rm ab}}

\newcommand{\ad}{{\rm ad}}

\newcommand{\FF}{{\mathbb F}}
\newcommand{\CC}{{\mathbb C}}
\newcommand{\RR}{{\mathbb R}}
\newcommand{\ZZ}{{\mathbb Z}}
\newcommand{\QQ}{{\mathbb Q}}
\newcommand{\NN}{{\mathbb N}}

\newcommand{\HH}{{\mathbb H}}

\newcommand{\GG}{{\mathbb G}}
\newcommand{\SSS}{{\mathbb S}}
\newcommand{\AAA}{{\mathbb A}}

\newcommand{\cG}{{\cal G}}
\newcommand{\cR}{{\cal R}}

\newcommand{\cL}{{\mathcal L}}

\newcommand{\OOO}{{\cal O}}

\newenvironment{prf}[1]{\trivlist
\item[\hskip \labelsep{\it
#1.\hspace*{.3em}}]}{~\hspace{\fill}~$\square$\endtrivlist}
\newenvironment{proof}{\begin{prf}{\bf Proof}}{\end{prf}}

\title{Categoricity of modular and Shimura curves}

\author{Christopher Daw and Adam Harris}

\date{\today}

\begin{document}
\maketitle

\begin{abstract}
We describe a model-theoretic setting for the study of Shimura varieties, and study the interaction between model theory and arithmetic geometry in this setting. In particular, we show that the model-theoretic statement of a certain $\cL_{\omega_1, \omega}$-sentence having a unique model of cardinality $\aleph_1$ is equivalent to a condition regarding certain Galois representations associated with Hodge-generic points. We then show that for modular and Shimura curves this $\cL_{\omega_1, \omega}$-sentence has a unique model in every infinite cardinality. In the process, we prove a new characterisation of the special points on any Shimura variety.
\end{abstract}

\section{Introduction}
Consider a connected component of a Shimura variety, which by definition is the quotient of a hermitian symmetric domain $X^+$ by a congruence subgroup $\Gamma$. Baily and Borel proved in \cite{BB} that automorphic forms on $X^+$ realise the quotient as a quasi-projective algebraic variety. This variety possesses a canonical model $S$ over some finite abelian extension of a number field $E$, where $E$ is independent of $\Gamma$ and referred to as the {\it reflex field}. We denote by $E^{\ab}$ the maximal abelian extension of $E$ and by $E^{\ab}(\Sigma)$ the field obtained by adjoining the coordinates of the images in $S(\CC)$ of the {\it special points} (see Section \ref{special} for the definition of a special point). Provided that $\Gamma$ is sufficiently small, $X^+$ is the universal cover of $S(\CC)$ and $\Gamma$ is the fundamental group. 

By a {\it modular curve} we refer to the connected components of Shimura varieties arising as quotients of the upper half-plane $\HH$ by congruence subgroups of $\GL_2(\QQ)$.  By a {\it Shimura curve} we refer to those quotients of $\HH$ by congruence subgroups of $G(\QQ)$, where $G:=\Res_{B/\QQ}\GG_{m,B}$ for a quaternion division algebra $B$ over $\QQ$ split at infinity. In other words,
\begin{align*}
G(R)=(B\otimes R)^{\times}
\end{align*}
for all $\QQ$-algebras $R$. Since $B$ splits over $\RR$, $G_\RR$ is isomorphic to $\GL_{2,\RR}$ and hence $G(\RR)^+$ acts on $\HH$. In both cases, $E=\QQ$ and we recall that $\QQ^{\ab}$ is the field extension of $\QQ$ obtained by adjoining all roots of unity. 

In general, one always has a reductive algebraic group $G$ defined over $\QQ$ such that $G^{\ad}(\RR)^+$ acts on $X^+$ as holomorphic isometries, where we write $G^{\ad}$ for $G$ modulo its centre. 
However, a principal theme of this paper is that a great deal of information is encoded in the action of the countable subgroup 
\begin{align*}
G^{\ad}(\QQ)^+:=G^{\ad}(\QQ)\cap G^{\ad}(\RR)^+.
\end{align*}
More precisely, consider a simplification of the above situation, given by
\begin{align*}
p:X^+\rightarrow S(\CC),
\end{align*}
where $X^+$ is now simply a set endowed with its action of $G^{\ad}(\QQ)^+$, $S$ is an algebraic variety defined over $E^{\ab}(\Sigma)$ and $p$ is a set-theoretic map that satisfies three key conditions that for now we will refer to as the {\it standard fibres} condition, the {\it special points} condition and the {\it modularity} condition (see Section \ref{axioms} for the special points and modularity conditions and Section \ref{cat} for the standard fibres condition). In this paper we prove that this information is enough to characterise $p$ up to isomorphism:

\begin{teo}\label{main}
Let $S$ be a modular or Shimura curve and suppose that
\begin{align*}
q:X^+\rightarrow S(\CC)
\end{align*}
satisfies the standard fibres, special points and modularity conditions. Then there exists a commutative diagram
\begin{center}
$\begin{CD}
X^+ @>\varphi>> X^+\\
@VVp V @VVq V\\
S(\CC) @>\sigma>> S(\CC),
\end{CD}$
\end{center}
where $\varphi$ is a $G^{\ad}(\QQ)^+$-equivariant bijection and $\sigma$ is an automorphism of $\CC$ fixing $E^{\ab}(\Sigma)$.
\end{teo} 

One can arguably phrase this problem more naturally in terms of model theory: consider two-sorted structures of the form
\begin{align*}
\langle\textbf{D},\textbf{S},q\rangle,
\end{align*}
where $\textbf{D}$ comprises a set endowed with a group action (see Section \ref{modeltheory} for more details), $\textbf{S}$ comprises an algebraic variety over a countable field $k$ with points in some algebraically closed field $F$, which we intepret in the field sort corresponding to $F$ with constants $k$, and $q$ is a function from $\mathbf{D}$ to $\mathbf{S}$. It follows that any connected component of a Shimura variety gives rise to such a structure. In a suitable language $\mathcal{L}$, we consider the complete first order theory of this structure and adjoin the $\mathcal{L}_{\omega_1,\omega}$ standard fibres condition, which says that the fibres of $q$ carry a transitive action of $\Gamma$. Theorem \ref{main} follows from the model-theoretic statement that this sentence is $\kappa$-{\it categorical} for all infinite cardinalities $\kappa$. Our Theorem \ref{mainmodel} is the statement that, if we additionally include the axiom that $F$ has infinite transcendence degree, this holds for any modular or Shimura curve. In Corollary \ref{com}, we prove that our axiomatisation for the first order theory of this structure is complete and satisfies quantifier elimination. 

{\it The key theme of this paper, however, is the interaction between model theory and classical results from arithmetic geometry.} The difficulty in proving Theorem \ref{main} is that mapping one point from $X^+$ to another maps its entire $G^{\ad}(\QQ)^+$ orbit (otherwise known as its {\it Hecke orbit}). This must be accounted for in the algebraic variety by a field automorphism, which is possible given certain open image theorems concerning the Galois representations associated with points on the Shimura variety or products thereof. Loosely speaking, these imply that a field automorphism can account for all but a finite number of constraints in the Hecke orbit. The remaining constraints are then accounted for via a judicious choice of the map in the first place. Particularly striking, however, is that categoricity actually implies these arithmetic statements via a model-theoretic result of Keisler (see \cite{K1}), restricting the number of types in models of $\aleph_1$-categorical $\mathcal{L}_{\omega_1,\omega}$-sentences. This implication is Theorem \ref{OI}.

Our work is motivated by earlier results of this type concerning the universal covers of the multiplicative group (see \cite{Z1} and \cite{BZ}), an elliptic curve (see \cite{Mi}) and semi-abelian varieties (see \cite{Z}), originating in a programme of Boris Zilber. Our proof follows the abstract strategy outlined in the aforementioned articles and relies on, essentially as a black box, the theory of quasi-minimal excellence as in \cite{BHHKK} and \cite{K}. We note that Gavrilovich first posed the question for Shimura curves in \cite{Mi}, Conjecture 9.

Our restriction to the dimension $1$ case is due to the abundance of Mumford-Tate groups that arise in higher dimensions. In dimension $1$, the only proper Mumford-Tate groups come from the special points and a key model-theoretic observation of this paper is that, in our setup, the special points always belong to the definable closure of the empty set. This follows from our Theorem \ref{dcl} that any special point is always the unique fixed point of some element of $G^{\ad}(\QQ)^+$. We feel that this result is of independent interest. In general, proper Mumford-Tate groups impose restrictions on the corresponding Galois representations and our language is simply not sophisticated enough to handle this additional level of complexity; in higher dimensional cases the model-theoretic setup will require refinement.

\section{Shimura varieties}

For further details regarding Shimura varieties, we refer the reader to \cite{M}. We denote by $\AAA_f$ the finite rational adeles and by $\hat{\ZZ}$ the product of $\ZZ_p$ over all primes $p$. 

Consider a Shimura datum $(G,X)$. By defintion, it consists of a reductive algebraic group $G$ over $\QQ$ and the $G(\RR)$-conjugacy class $X$ of a homomorphism
\begin{align*}
 h:\SSS\rightarrow G_{\RR}
\end{align*}
satisfying the three axioms SV1-SV3 of \cite{M}, p50. Let $X^+$ be a connected component of $X$; it is known that its stabiliser in $G^{\ad}(\RR)$ is $G^{\ad}(\RR)^+$. Let $K$ be a compact open subgroup of $G(\AAA_f)$ and let $\cal{C}$ be a set of representatives for the finite set
\begin{align*}
G(\QQ)_+\backslash G(\AAA_f)/K,
\end{align*}
where $G(\QQ)_+$ is the inverse image of $G^{\ad}(\QQ)^+$ in $G(\QQ)$ i.e. the stabiliser of $X^+$ in $G(\QQ)$. Note that the action of $G(\QQ)_+$ on $X^+$ factors through $G^{\ad}(\QQ)^+$. Therefore, the elements of $G^{\ad}(\QQ)^+$ may be thought of as functions from $X^+$ to itself.

The double coset space
\begin{align*}
\Sh_K(G,X):=G(\QQ)\backslash(X\times (G(\AAA_f)/K)),
\end{align*}
where $G(\QQ)$ acts diagonally, is equal to the finite disjoint union 
\begin{align*}
\coprod_{g\in\cal{C}}\Gamma_g\backslash X^+,
\end{align*}
where $\Gamma_g:= G(\QQ)_+\cap gKg^{-1}$ and, when the $\Gamma_g$ are sufficiently small, the members of this union have canonical realisations as quasi-projective varieties over $\CC$. In fact, by \cite{M}, Theorem 3.12, the $\Gamma_g$ are sufficiently small when their images in $G^{\ad}(\QQ)^+$ are torsion free. Therefore, $\Sh_K(G,X)$ may always be realised as a quasi-projective variety over $\CC$ and, by work of Shimura, Deligne, Borovoi and Milne, it possesses a canonical model over its reflex field $E:=E(G,X)$.

Now let $\Gamma$ denote $\Gamma_g$ for some $g\in\mathcal{C}$ and consider the connected component $\Gamma\backslash X^+$. We lose no generality if we assume that $g$ is the identity. Let $S$ be the corresponding irreducible component of the canonical model, which is always defined over $E^{\ab}$. Let
\begin{align*}
p:X^+\rightarrow S(\CC)
\end{align*}
denote the surjective map from $X^+$ to $S(\CC)$ given by automorphic forms. 

\begin{conv}
By a Shimura variety we henceforth refer to an $S$ as constructed above. Furthermore, we will assume that the image of $\Gamma$ in $G^{\ad}(\QQ)^+$ is torsion free until we address this case in Section \ref{red}.
\end{conv}

\subsection{Modular curves}

For $G$ take $\GL_2$ and for $X$ take the $\GL_2(\RR)$-conjugacy class of the morphism $h:\SSS\rightarrow\GL_{2,\RR}$ defined by
\begin{align*}
a+ib\mapsto\left(\begin{array}{cc} a & -b\\ b & a\end{array}\right).
\end{align*}
This gives a Shimura datum and the corresponding reflex field is $\QQ$, whose maximal abelian extension is the field obtained by adjoining all roots of unity.

The set $X$ is naturally in bijection with the union of the upper and lower half planes in $\CC$ and for $X^+$ we take the upper half plane $\HH$. The stabiliser of $\HH$ in $\GL_2(\QQ)$ is the subgroup $\GL_2(\QQ)^+$ of matrices with positive determinant. 

For $K$ we take any compact open subgroup of $\GL_2(\AAA_f)$, which we assume to be contained in $\GL_2(\hat{\ZZ})$. The quotient of $\HH$ by $\Gamma:=K\cap\GL_2(\QQ)^+$ has the structure of an algebraic curve with a canonical model $S$ over $\QQ^{\ab}$. We refer to $S$ as a {\it modular curve}. Consider the following choices for $K$:

\begin{enumerate}

\item[(1)] $K(N):=\prod_pK(N)_p$, where
\begin{align*}
K(N)_p:=\left\{g\in\GL_2(\mathbb{Z}_p)\mid g\equiv \left(\begin{array}{cc} 1 & 0\\ 0 & 1\end{array}\right)\hbox{ mod }p^{{\rm ord}_p(N)}\right\}.
\end{align*}
We have $K(N)\cap\GL_2(\QQ)=\Gamma(N)$, where
\begin{align*}
\Gamma(N)=\left\{\gamma\in\SL_2(\ZZ)\mid \gamma\equiv \left(\begin{array}{cc} 1 & 0\\ 0 & 1\end{array}\right)\hbox{ mod }N\right\},
\end{align*}
the principal congruence subgroup of level $N$ (note that any compact open subgroup $K$ of $\GL_2(\AAA_f)$ contains a $K(N)$ with finite index). Then $S(\CC)$ is the moduli space for isomorphism classes of elliptic curves with level-$N$ structure. 

\item[(2)] $K_0(N):=\prod_pK_0(N)_p$, where
\begin{align*}
K_0(N)_p:=\left\{g\in\GL_2(\mathbb{Z}_p)\mid g\equiv \left(\begin{array}{cc} * & *\\ 0 & *\end{array}\right)\hbox{ mod }p^{{\rm ord}_p(N)}\right\}.
\end{align*}
We have $K_0(N)\cap\GL_2(\QQ)=\Gamma_0(N)$, where
\begin{align*}
\Gamma_0(N)=\left\{\gamma\in\SL_2(\ZZ)\mid \gamma\equiv \left(\begin{array}{cc} * & *\\ 0 & *\end{array}\right)\hbox{ mod }N\right\}.
\end{align*}
Then $S(\CC)$ is the moduli space for isomorphism classes of elliptic curves with a cyclic subgroup of order $N$.

\item[(3)] $K_1(N):=\prod_pK_1(N)_p$, where
\begin{align*}
K_1(N)_p:=\left\{g\in\GL_2(\mathbb{Z}_p)\mid g\equiv \left(\begin{array}{cc} 1 & *\\ 0 & 1\end{array}\right)\hbox{ mod }p^{{\rm ord}_p(N)}\right\}.
\end{align*}
We have $K_1(N)\cap\GL_2(\QQ)=\Gamma_1(N)$, where
\begin{align*}
\Gamma_1(N)=\left\{\gamma\in\SL_2(\ZZ)\mid \gamma\equiv \left(\begin{array}{cc} 1 & *\\ 0 & 1\end{array}\right)\hbox{ mod }N\right\}.
\end{align*}
Then $S(\CC)$ is the moduli space for isomorphism classes of elliptic curves with an element of order $N$.

\item[(4)] When $N=1$ we have $K(N)=\GL_2(\hat{\ZZ})$, $\Gamma(N)=\SL_2(\ZZ)$ and $S=\AAA^1$. Furthermore, 
\begin{align*}
p:\HH\rightarrow\CC
\end{align*}
is the classical $j$-function. Of course, while we are imposing the torsion free condition on $\Gamma$, we are forced to assume that $N>1$. By \cite{M2}, p39, this condition is also sufficient.

\end{enumerate}

\subsection{Shimura curves}

Let $B$ be a quaternion division algebra over $\QQ$ split at infinity. Then $B$ splits at almost all primes $p$, namely those not dividing its discriminant. Let $G$ be the algebraic group defined by
\begin{align*}
G(R)=(B\otimes R)^{\times}
\end{align*}
for all $\QQ$-algebras $R$. Alternatively, we could write $G:=\Res_{B/\QQ}\GG_{m,B}$. 

We have that $G_{\QQ_p}$ is isomorphic to $\GL_{2,\QQ_p}$ for almost all primes. Since $B$ splits over $\RR$, $G_\RR$ is isomorphic to $\GL_{2,\RR}$ and so $G(\RR)$ acts on the union of the upper and lower half-planes. Therefore, $X$ and $X^+$ are taken as in the previous section and $(G,X)$ is again a Shimura datum.

Choosing a compact open subgroup $K$ of $G(\AAA_f)$, we obtain a congruence subgroup $\Gamma:=K\cap G(\QQ)_+$. The quotient $\Gamma\backslash\HH$ has the structure of a compact algebraic curve with a canonical model $S$. We refer to $S$ as a {\it Shimura curve}. It parametrises abelian varieties of dimension two with quaternionic multiplication and additional data. This is potentially a subset of a more general definition of Shimura curves since we have restricted to quaternion algebras over $\QQ$.

\subsection{Special points}\label{special}

Recall the definition of a special point:

\begin{defini}\label{sp}
Let $(G,X)$ be a Shimura datum. A point $x\in X$ is called special if there exists a subtorus $T$ of $G$ defined over $\QQ$ such that 
\begin{align*}
x:\SSS\rightarrow G_\RR
\end{align*}
factorises through $T_\RR$.
\end{defini}

We believe the following is a new characterisation of special points. It is essential for the purposes of this paper, but we also believe it should be interesting in its own right. For example, it implies that in certain realisations of $X^+$ special points are always algebraic (see \S1, \cite{DO}).

\begin{teo}\label{dcl}
Let $(G,X)$ be a Shimura datum and let $X^+$ be a connected component of $X$. For any $x\in X^+$ the following are equivalent:
\begin{enumerate}
\item[(1)] The point $x$ is special.
\item[(2)] There exists a $g\in G^{\ad}(\QQ)^+$ with the property that $x$ is the unique fixed point of $g$ in $X^+$.
\end{enumerate}
\end{teo}

\begin{proof}
Firstly, assume that (2) holds for $h\in G^{\ad}(\QQ)^+$ and let $g\in G(\QQ)$ be a lift of $h$. Let $M:=\MT(x)$ be the Mumford-Tate group of $x$ i.e. the smallest algebraic subgroup $H$ of $G$ defined over $\QQ$ such that $x$ factors though $H_\RR$. By \cite{UY}, Lemma 2.2, since $g$ fixes $x$, $g$ belongs to the centraliser of $M(\QQ)$ in $G(\QQ)$. Let $X_M$ be the $M(\RR)$-conjugacy class of $x$ and let $X^+_M$ be a connected component of $X_M$ contained in $X^+$. Since $g$ fixes every element of $X^+_M$, it must contain only one element. Therefore, $M$ is commutative and, by the general theory of Shimura varieties, reductive. Hence, $M$ is a torus and $x$ is special.

Now assume that (1) holds. Therefore, 
\begin{align*}
x:\SSS\rightarrow G_{\RR}
\end{align*} 
factors through $T_{\RR}$, where $T$ is a maximal torus of $G$, defined over $\QQ$. Let $\pi:G\rightarrow G^{\ad}$ be the natural morphism and let $S:=\pi(T)$. Let $K_{\infty}$ be the connected maximal compact subgroup of $G^{\ad}(\RR)^+$ constituting the stabiliser of $x$. Then $S(\RR)$ is a maximal torus of $K_{\infty}$. 

Suppose $S(\RR)$ fixes another point $gx\in X^+$ for some $g\in G^{\ad}(\RR)^+$ (note that $K_{\infty}$ is also the stabiliser of $x$ in $G^{\ad}(\RR)$ but it is not necessarily a maximal compact subgroup of $G^{\ad}(\RR)$; consider the difference between $O(2)$ and $SO(2)$, for example. Hence, the theorem does not extend to $X$ i.e we cannot take $g\in G^{\ad}(\RR)$). Then $g^{-1}S(\RR)g$ is contained in $K_{\infty}$ and, since all maximal tori in connected compact lie groups are conjugate, there exists $k\in K_{\infty}$ such that 
\begin{align*}
(gk)^{-1}S(\RR)(gk)=S(\RR).
\end{align*} 
Therefore, $gk$ belongs to the normaliser $N$ of $S(\RR)$ in $G^{\ad}(\RR)^+$.

Note that, since $S$ is maximal and $G^{\ad}$ is reductive, the centraliser of $S$ in $G^{\ad}$ is $S$ itself. Therefore, the quotient $N/S(\RR)$ is the (finite) Weyl group of $S(\RR)$ in $G^{\ad}(\RR)^+$. Hence, $N$ is also compact and, since all maximal compact subgroups in real lie groups are conjugate, it must be contained in $hK_{\infty}h^{-1}$ for some $h\in G^{\ad}(\RR)^+$.

Therefore, $h^{-1}S(\RR)h$ is contained in $K_{\infty}$ and, since all maximal tori in connected compact lie groups are conjugate, there exists a $k\in K_{\infty}$ such that
\begin{align*}
(hk)^{-1}S(\RR)(hk)=S(\RR).
\end{align*}
Thus, $hk\in N$ and so $h=nk$ for some $n\in N$. Therefore, $hK_{\infty}h^{-1}$ is equal to $nK_{\infty}n^{-1}$ and $N$ is contained in $K_{\infty}$. We conclude that $g\in K_{\infty}$ and $x$ is the unique point of $X^+$ fixed by $S(\RR)$.

Now let $s\in S(\QQ)$ be a regular element. By this we mean an element $s\in S(\QQ)$ such that the centraliser $Z_{G^{\ad}}(s)$ of $s$ in $G^{\ad}$ is equal to $S$. These are precisely the elements not belonging to $\ker(\alpha)$ for any of the finitely many roots $\alpha$ of $S$. Note that, since $S$ is maximal, the union of these subvarieties and their Galois conjugates is a proper closed subvariety of $S$. Hence, regular elements in $S(\QQ)$ do exist.

Suppose that $s$ fixes a point
\begin{align*}
x':\SSS\rightarrow G_{\RR}
\end{align*} 
in $X^+$ i.e. $\ad(s)(t)=t$ for all $t\in x'(\SSS)(\RR)$. Therefore, $\pi(t)s\pi(t)^{-1}=s$ for all $t\in x'(\SSS)(\RR)$ and so $\pi(x'(\SSS))(\RR)$ is contained in $Z_{G^{\ad}}(s)(\RR)=S(\RR)$. Therefore, $S(\RR)$ fixes $x'$ and so $x'=x$. 

\end{proof}



\subsection{Hodge-generic points}

Recall the definition of a Hodge-generic point:

\begin{defini}
Let $(G,X)$ be a Shimura datum. A point $x\in X$ is called Hodge-generic if 
\begin{align*}
x:\SSS\rightarrow G_\RR
\end{align*} 
does not factor through $H_\RR$ for any proper algebraic subgroup $H$ of $G$ defined over $\QQ$.
\end{defini}

The following property of Hodge-generic points follows immediately from \cite{UY}, Lemma 2.2:

\begin{prop}\label{HG}
Let $(G,X)$ be a Shimura datum and let $X^+$ be a connected component of $X$.  Let $x\in X^+$ be Hodge-generic and let $g\in G^{\ad}(\QQ)^+$ fix $x$. Then $g$ is the identity.
\end{prop}

In this paper we will restrict our attention to Shimura varieties of dimension $1$ for the reason that, in this situation, every point is either special or Hodge-generic. To see this, note that any proper Mumford-Tate group gives rise to a special subvariety (see, for example, \cite{M}, Theorem 5.16). In particular, our models will satisfy {\it uniqueness of the generic type} defined in \cite{BHHKK}, Definition 2.1, QM4.

\subsection{Galois representations}

Let $S$ be a Shimura variety. 

\begin{lem}\label{nsp}
Let $\bar{g}:=(g_1,\ldots,g_n)$ be a tuple of elements belonging to $G^{\ad}(\QQ)^+$. The image of the map \begin{align*}
f:X^+\rightarrow S(\CC)^n
\end{align*} 
sending $x$ to $(p(g_1x),\ldots,p(g_nx))$ is an algebraic variety defined over $E^{\ab}$, which we denote $Z_{\bar{g}}$.
\end{lem}

\begin{proof}
Consider the compact open subgroup 
\begin{align*}
K':=g_1Kg^{-1}_1\cap\cdots\cap g_nKg^{-1}_n.
\end{align*} 
of $G(\AAA_f)$. The image of $f$ is an irreducible component of the image of the map
\begin{align*}
\Sh_{K'}(G,X)\rightarrow\Sh_K(G,X)^n
\end{align*}
induced by the morphism of Shimura data coming from
\begin{align*}
G\rightarrow G^n:g\mapsto (g_1gg^{-1}_1,\ldots,g_ngg^{-1}_n).
\end{align*}
By \cite{M}, Theorem 13.6, this map is defined over $E$ and the connected components of $\Sh_{K'}(G,X)$ are defined over $E^{\ab}$.

\end{proof}

Consider the subvariety $Z_{\bar{g}}$, where $\bar{g}:=(e,g_1,\ldots,g_n)$ is a tuple of distinct elements belonging to $G^{\ad}(\QQ)^+$ and $e$ denotes the identity element. It is biholomorphic to $\Gamma_{\bar{g}}\backslash X^+$, where 
\begin{align*}
\Gamma_{\bar{g}}:=\Gamma\cap g^{-1}_1\Gamma g_1\cap\cdots\cap g^{-1}_n\Gamma g_n. 
\end{align*}

Consider a subset of $\{g_1,\ldots,g_n\}$, which we denote $\{h_1,\ldots,h_m\}$ for some $m$ at most $n$. We obtain another tuple ${\bar{h}}:=(e,h_1,\ldots,h_m)$ and a finite morphism $Z_{\bar{g}}\rightarrow Z_{\bar{h}}$ via the natural map
\begin{align*}
\Gamma_{\bar{g}}\backslash X^+\rightarrow\Gamma_{\bar{h}}\backslash X^+
\end{align*}
whose fibres carry a transitive action of $\Gamma_{\bar{h}}$. If $\Gamma_{\bar{g}}$ is normal in $\Gamma_{\bar{h}}$ then this action factors through the group $\Gamma_{\bar{h}}/\Gamma_{\bar{g}}$. On $Z_{\bar{g}}$ this action is given by regular maps defined over $E^{\ab}$ so, if $z\in Z_{\bar{h}}$ has coordinates in an extension $L$ of $E^{\ab}$, the action of $\Aut(\CC/L)$ on the fibre above $z$ commutes with the action of $\Gamma_{\bar{h}}/\Gamma_{\bar{g}}$. This action is transitive and  (by virtue of $\Gamma$ being torsion free in $G^{\ad}(\QQ)^+$) without fixed points. In particular, the fibre above $z$ is a torsor for the group $\Gamma_{\bar{h}}/\Gamma_{\bar{g}}$. Therefore, for any point in the fibre we obtain a continuous homomorphism
\begin{align*}
\Aut(\CC/L)\rightarrow\Gamma_{\bar{h}}/\Gamma_{\bar{g}}
\end{align*}
as in \cite{P}, \S6. For any two such points $\tilde{z}$ and $\tilde{z}'$, there exists a unique element $\gamma\in\Gamma_{\bar{h}}/\Gamma_{\bar{g}}$ such that $\tilde{z}'=\gamma\tilde{z}$ and the corresponding homomorphisms are conjugate by this $\gamma$.

We obtain an inverse system of varieties $Z_{\bar{g}}$ with finite morphisms to $S(\CC)$, indexed by tuples ${\bar{g}}=(e,g_1,\ldots,g_n)$ of distinct elements in $G^{\ad}(\QQ)^+$ such that $\Gamma_{\bar{g}}$ is a normal subgroup of $\Gamma$. Note that, for any $g\in G(\QQ)$, the double coset $\Gamma g\Gamma$ is the disjoint union of finitely many single cosets $h\Gamma$, where $h\in G(\QQ)$ (see \cite{M2}, Lemma 5.29). It follows that this system of normal subgroups is coinitial in the system coming from all tuples. 

The limit 
\begin{align*}
\overline{\Gamma}:=\varprojlim_{\bar{g}} \Gamma/\Gamma_{\bar{g}}
\end{align*} 
over all tuples ${\bar{g}}=(e,g_1,\ldots,g_n)$ such that $\Gamma_{\bar{g}}$ is normal in $\Gamma$ again acts freely and transitively on the fibre above a point in $S(\CC)$. Therefore, if $z\in S(\CC)$ has coordinates in an extension $L$ of $E^{\ab}$, we obtain (the conjugacy class of) a continuous homomorphism
\begin{align*}
\Aut(\CC/L)\rightarrow\overline{\Gamma}.
\end{align*}
Similarly, if we take a tuple of points ${\bar{z}}:=(z_1,\ldots,z_m)\in S(\CC)^m$ with coordinates in an extension $L$ of $E^{\ab}$, the action of $\Aut(\CC/L)$ on the fibre above ${\bar{z}}$ is given by (the conjugacy class of) a continuous homomorphism
\begin{align*}
\Aut(\CC/L)\rightarrow\overline{\Gamma}^m.
\end{align*}
We note here that when $S(\CC)$ is a moduli space for abelian varieties and $A$ is the (isomorphism class of an) abelian variety defined over $L$ associated to $\bar{z}$ then the above homomorphism is obtained from the action of $\Aut(\CC/L)$ on the adelic Tate module of $A$. See \cite{UY} for further details. 

\section{Model theory}\label{modeltheory}

In this paper we will consider three types of model-theoretic structures, two of which will be one-sorted and the other two-sorted.  

Let $S$ be a Shimura variety. When we refer to the \textit{covering structure}, we refer to a one-sorted structure of the form
\begin{align*}
\textbf{D}: = \langle D , \{ f_g \}_{g \in G^{\ad}(\QQ)^+}\rangle,
\end{align*}
where $D$ is a set and each $f_g:D \rightarrow D$ is a unary function. These functions will correspond to a left action of $G^{\ad}(\QQ)^+$ on $D$ and we will write $g$ instead of $f_g$ and $G^{\ad}(\QQ)^+$ instead of $\{ f_g \}_{g \in G^{\ad}(\QQ)^+}$. In particular, we will write a covering structure as
\begin{align*}
\textbf{D} = \langle D, G^{\ad}(\QQ)^+ \rangle.
\end{align*}
We denote by $\mathcal{L}_{G}$ the natural language for these one-sorted structures.

When we refer to the \textit{variety structure}, we refer to a one-sorted structure of the form
\begin{align*}
\textbf{S}:=\langle S(F), \cR \rangle
\end{align*}
where $S(F)$ are the points of $S$ in an algebraically closed field $F$ and $\cR:=\cR(F,E^{\ab}(\Sigma))$ is a set containing a relation for every Zariski closed subset of $S(F)^n$ defined over $E^{\ab}(\Sigma)$ for all $n \in \NN$. However, we will interpret this structure in the field sort
\begin{align*}
\textbf{F}:=\langle F,+,\cdot,E^{\ab}(\Sigma)\rangle,
\end{align*}
namely the algebraically closed field $F$ with constants added for each element of $E^{\ab}(\Sigma)$. Therefore, we will use the language of rings $\mathcal{L}_r$ for the variety structure.

Finally, when we refer to the {\it two-sorted structure}, we refer to a two-sorted structure of the form
\begin{align*}
{\bf q}:=\langle{\bf D},{\bf S},q\rangle,
\end{align*}
where ${\bf D}$ is a covering structure, ${\bf S}$ is a variety structure and $q$ is a function from ${\bf D}$ to ${\bf S}$. We denote by $\mathcal{L}$ the natural language for these two-sorted structures, namely the combination of $\mathcal{L}_{G}$, $\mathcal{L}_r$ and a function symbol between the two sorts.\\

Clearly, $S$ determines a two-sorted structure
\begin{align*}
{\bf p}:=\langle\langle X^+,G^{\ad}(\QQ)^+\rangle,\langle S(\CC),\cR\rangle,p:X^+\rightarrow S(\CC)\rangle.
\end{align*}
By Theorem \ref{dcl}, each special point in $X^+$ belongs to the definable closure of the empty set. This is why we include $E^{\ab}(\Sigma)$ as our field of constants rather than $E^{\ab}$. Note that there are only countably many special points since the coordinates of their images in $S(\CC)$ are algebraic. 

Let $\Th \langle X^+, G^{\ad}(\QQ)^+ \rangle$ be the complete first order $\mathcal{L}_{G}$-theory of the covering structure $\langle X^+, G^{\ad}(\QQ)^+ \rangle$. Note that, for any group $\cG$, the class of faithful $\cG$-sets is first order axiomatisable. In particular, any model of $\Th \langle X^+, G^{\ad}(\QQ)^+ \rangle$ is indeed a $G^{\ad}(\QQ)^+$-set and, furthermore, has well-defined notions of {\it special point} and {\it Hodge-generic point}: a point that is the unique fixed point of an element of $G^{\ad}(\QQ)^+$ and a point fixed by no element of $G^{\ad}(\QQ)^+$, respectively.

\begin{conv}
Let $\mathcal{L}_0$ be a language and let ${\bf M}$ be an $\mathcal{L}_0$-structure with universe $M$. If $A\subseteq M$ and $\bar{a}:=(a_1,...,a_n)\in M^n$ then we denote the quantifier-free type of $\bar{a}$ over $A$ by
\begin{align*}
{\rm qftp}_{\mathcal{L}_0}(\bar{a}/A).
\end{align*}
\end{conv}

\subsection{Axiomatisation}\label{axioms}

Let $S$ be a Shimura variety, let ${\bf p}$ be the corresponding two-sorted structure and let $\Th({\bf p})$ denote the complete first order $\mathcal{L}$-theory of ${\bf p}$. For each of the subvarieties $Z_{\bar{g}}$ and corresponding tuple ${\bar{g}}:=(g_1,\ldots,g_n)$ of elements in $G^{\ad}(\QQ)^+$, consider the following $\cL$-sentences appearing in $\Th(\bf p)$:
\begin{enumerate}
\item[(1)] $\textsf{MOD}^1_{\bar{g}}:= \forall x\in D\ (q(g_1x),\ldots,q(g_nx))\in Z_{\bar{g}}$
\item[(2)] $\textsf{MOD}^2_{\bar{g}}:= \forall z\in Z_{\bar{g}}\ \exists x\in D\ (q(g_1x),\ldots,q(g_nx))=z$
\end{enumerate}
Define \textit{the modularity} condition to be the axiom scheme 
\begin{align*}
\textsf{MOD} := \bigcup_{\bar{g}}\textsf{MOD}^1_{\bar{g}}\wedge\textsf{MOD}^2_{\bar{g}}.
\end{align*}
For each special point $x\in X^+$, choose $g_x\in G^{\ad}(\QQ)^+$ that fixes $x$ and only $x$ in $X^+$ (whose existence is ensured by Theorem \ref{dcl}). Consider the following $\cL$-sentences appearing in $\Th(\bf p)$:
\begin{align*}
\textsf{SP}_x := \forall y\in D\ \left(g_xy=y \implies q(y)=p(x)\right).
\end{align*}
The {\it special points} condition is defined as the axiom scheme
\begin{align*}
\textsf{SP}:= \bigcup_x \textsf{SP}_x. 
\end{align*}
Now let ${\rm T}({\bf p})$ denote the following union of $\mathcal{L}$-sentences appearing $\Th(\bf p)$:
\begin{align*}
{\rm T}({\bf p}):= \Th \langle X^+,G^{\ad}(\QQ)^+\rangle \cup \Th \langle S(\CC),\cR \rangle \cup \textsf{MOD} \cup \textsf{SP},
\end{align*}
where $\Th \langle S(\CC),\cR \rangle$ denotes the complete first-order $\cL_r$-theory of $\langle S(\CC),\cR \rangle$.

\subsection{Quantifier elimination and completeness}

Consider a Shimura variety $S$ and let $\bf p$ be the corresponding two-sorted structure. Suppose that we have two two-sorted structures
\begin{align*}
{\bf q}:=\langle{\bf D},{\bf S},q\rangle,\ {\bf q'}:=\langle{\bf D'},{\bf S'},q'\rangle,
\end{align*}
which are models of ${\rm T}({\bf p})$, and suppose that
\begin{align*}
\rho:D\cup S(F)\rightarrow D'\cup S(F')
\end{align*}
is a partial isomorphism with finitely generated domain $U$. By definition, $U$ is the union of $U_{\bf D}:=U\cap D$ and $U_{\bf S}:=U\cap S(F)$, where $U_{\bf D}$ is the union of the $G^{\ad}(\QQ)^+$ orbits of finitely many $x\in D$ and $U_{\bf S}$ is $S(L)$ for some field $L$ generated by the coordinates of the images of the $x\in U_{\bf D}$ along with finitely many other points in $S(F)$. Therefore, $\rho$ restricts to a $G^{\ad}(\QQ)^+$-equivariant injection $U_{\bf D}\rightarrow D'$ and an embedding $S(L)\rightarrow S(F')$ induced by an embedding of $L$ into $F'$ fixing $E^{\ab}(\Sigma)$.

\begin{prop}\label{tp2}
Suppose that $x\in D$ is a special point such that $x\notin U_{\bf D}$. There exists a (unique) realisation of $\rho({\rm qftp}_{\mathcal{L}}(x/U))$ in $D'$.
\end{prop}

\begin{proof}
Since $x$ is special, there exists only one possible choice, which we denote $x'$. Now consider an atomic formula $f$ in ${\rm qftp}_{\mathcal{L}}(x/U)$. If $f$ can be written in $\mathcal{L}_G$ then it can only be $x=gx$, $x\neq gx$ or $x\neq y$, where $g\in G^{\ad}(\QQ)^+$ and $y\in U_{\bf D}$. In which case, $x'$ must clearly satisfy $\rho(f)$. Otherwise, $f$ is an formula in $q(g_1x),\ldots,q(g_nx)$, where $g_1,\ldots,g_n\in G^{\ad}(\QQ)^+$, with parameters in $L$. Since $\bf q$ and ${\bf q'}$ satisfy $\mathsf{SP}$, we know that $q'(g_i(x'))=q(g_i(x))$ and has coordinates in $E^{\ab}(\Sigma)$ for all $i\in\{1,\ldots n\}$. Therefore, $q'(g_1x'),\ldots,q'(g_nx')$ satisfy $\rho(f)$ and the result follows.
\end{proof}

\begin{prop}\label{tp}
Suppose that $x\in D$ is a Hodge-generic point such that $x\notin U_{\bf D}$. Then ${\rm qftp}_{\mathcal{L}}(x/U)$ is determined by 
\begin{align*}
\bigcup_{\bar{g}}{\rm qftp}_{\mathcal{L}}((q(g_1x),\ldots,q(g_nx))/L)
\end{align*}
ranging over all tuples $\bar{g}:=(g_1,\ldots,g_n)$ of elements of $G^{\ad}(\QQ)^+$. 
\end{prop}

\begin{proof}
Consider the non-trivial atomic formulae in ${\rm qftp}_{\mathcal{L}}(x/U)$ that can be written in $\mathcal{L}_G$. These are equivalent to $x\neq y$ for all $y\in U_{\bf D}$, and $x\neq gx$ for all non-trivial $g\in G^{\ad}(\QQ)^+$. Since $\Gamma$ acts without fixed points, these formulae follow from the following formulae appearing in the above union:
\begin{enumerate}
\item[(1)] $(q(x),q(y))\notin Z_{(e,e)}$, $\forall y\in U_{\bf D}$
\item[(2)] $(q(x),q(x))\notin Z_{(e,g)}$, $\forall g\notin\Gamma$
\end{enumerate}
where by $q(y)$ in (1) we refer to the coordinates in $S(L)$.
\end{proof}

\begin{prop}\label{real}
Suppose that $x\in D$ and let $\bar{g}:=(g_1,\ldots,g_n)$ denote a tuple of elements belonging to $G^{\ad}(\QQ)^+$. There exists an $x'\in D'$ such that 
the tuple
\begin{align*}
(q'(g_1x'),\ldots,q'(g_nx'))\in S(F')^n
\end{align*}
is a realisation of
\begin{align*}
\rho({\rm qftp}_{\mathcal{L}}((q(g_1x),\ldots,q(g_nx))/L))
\end{align*}
\end{prop}

\begin{proof}
Note that
\begin{align*}
{\rm qftp}_{\mathcal{L}}((q(g_1x),\ldots,q(g_nx))/L)
\end{align*} 
is determined by the smallest algebraic subvariety of $S(F)^n$ defined over $L$ containing the tuple in question. Since the axiom $\textsf{MOD}^1_{\bar{g}}$ is true in ${\bf q}$, this subvariety must be contained in $Z_{\bar{g}}$. Therefore, since $Z_{\bar{g}}^{\sigma}=Z_{\bar{g}}$, the algebraic subvariety that determines
\begin{align*}
\rho({\rm qftp}_{\mathcal{L}}((q(g_1x),\ldots,q(g_nx))/L)).
\end{align*}
is also contained in $Z_{\bar{g}}$. The result follows from the fact that $\textsf{MOD}^2_{\bar{g}}$ is true in ${\bf q'}$. 
\end{proof}

\begin{prop}{(cf. \cite{Z}, Lemma 3.4)}\label{qec}
Suppose that $S$ has dimension $1$ and that ${\bf q}$ and ${\bf q'}$ are $\omega$-saturated. Given any $\alpha\in D\cup S(F)$, $\rho$ extends to the substructure generated by $U\cup\{\alpha\}$.
\end{prop}

\begin{proof}
First consider the case that $\alpha\in S(F)$. We may assume that $\alpha\notin q(U_{\bf D})$ since otherwise we have $\alpha\in U_{\bf S}$. Since ${\rm qftp}_{\mathcal{L}}(\alpha/U)$ is determined by ${\rm qftp}_{\mathcal{L}_r}(\alpha/L)$, we can extend $\rho$ by choosing a realisation of $\rho({\rm qftp}_{\mathcal{L}_r}(\alpha/L))$.

Now consider the case that $\alpha\in D$ such that $\alpha\notin U_{\bf D}$. If $x$ is special then $\rho(x)$ is prescribed by Proposition \ref{tp2}. Therefore, since $S$ has dimension $1$, we may assume that $x$ is Hodge-generic. By Proposition \ref{tp}, $\rho({\rm qftp}_{\mathcal{L}}(x/U))$ is determined by 
\begin{align*}
\bigcup_{\bar{g}}\rho({\rm qftp}_{\mathcal{L}}((q(g_1x),\ldots,q(g_nx))/L))
\end{align*}
ranging over all tuples $\bar{g}:=(g_1,\ldots,g_n)$ of elements of $G^{\ad}(\QQ)^+$. By Proposition \ref{real}, every finite subset of this type is realisable in a model of ${\rm T}({\bf p})$. Therefore, by compactness, this type is consistent and, since ${\bf q'}$ is $\omega$-saturated, it has a realisation $x'\in D'$. 
\end{proof}

\begin{cor}\label{com}
The theory ${\rm T}({\bf p})$ has quantifier elimination, is complete, model complete and superstable.
\end{cor}

\begin{proof}
See \cite{Pil}, Proposition 2.29, Proposition 2.30 and Remark 2.38. Superstability follows easily from quantifier elimination.
\end{proof}

\section{Categoricity}\label{cat}

Let $S$ be a Shimura variety, let $\bf p$ denote the corresponding two-sorted structure. We define the {\it standard fibres} condition to be the ${\cal{L}}_{\omega_1,\omega}$-axiom
\begin{align*}
\mathsf{SF}:=\forall x\forall y\in D\left(q(x)=q(y)\rightarrow\bigvee_{\gamma\in\Gamma}x=\gamma y\right).
\end{align*} 
We denote by $\Th^{\infty}_{\textsf{SF}}({\bf p})$ the union of $\Th({\bf p})$, the standard fibres condition and the additional axiom that the transcendence degree of $F$ is infinite.

Now assume that $S$ is a modular or Shimura curve. Our objective is to prove Theorem \ref{main}. This follows from the model-theoretic statement that the theory $\Th^{\infty}_{\textsf{SF}}({\bf p})$ is $\kappa$-{\it categorical} for all infinite cardinalities $\kappa$. 

\subsection{Necessary conditions}

Let $S$ be a Shimura variety and let $\bf p$ be the corresponding two-sorted structure. Let $x_1,\ldots,x_m\in X^+$ be a collection of Hodge-generic points in distinct $G^{\ad}(\QQ)^+$ orbits and let $L$ be the field obtained from $E^{\ab}(\Sigma)$ by adjoining the coordinates of $\bar{z}:=(p(x_1),\ldots,p(x_m))$. We will prove the following theorem:

\begin{teo}\label{OI}
If $\Th^{\infty}_{\textsf{SF}}({\bf p})$ is $\aleph_1$-categorical, then the image of the homomorphism
\begin{align*}
\Aut(\CC/L)\rightarrow\overline{\Gamma}^m
\end{align*}
associated with $\bar{z}$ has finite index.
\end{teo}
As previously mentioned, although the index is well-defined, the homomorphism itself is only defined up to conjugation. The key ingredient in the proof will be the following result of Keisler:

\begin{teo}{\bf(Keisler)}
If an $\mathcal{L}_{\omega_1,\omega}$-sentence is $\aleph_1$-categorical then the set of complete $m$-types realisable over the empty set in models of this sentence is at most countable.
\end{teo}

Therefore, we will show that if the above index is not finite then we are able to realise uncountably many distinct complete $m$-types in models of $\Th^{\infty}_{\textsf{SF}}({\bf p})$. To that end, consider the groups 
\begin{align*}
\Gamma_{\bar{g}}:=g^{-1}_1\Gamma g_1\cap\cdots\cap g^{-1}_n\Gamma g_n,
\end{align*}
for all tuples ${\bar{g}}:=(g_1,\ldots,g_n)$ of distinct elements in $G^{\ad}(\QQ)^+$. The quotients of $X^+$ by the groups $\Gamma_{\bar{g}}$ form an inverse system of locally symmetric varieties. The morphisms are those induced by inclusions of groups. Note that this system carries an action of $G^{\ad}(\QQ)^+$ since the action of $\alpha\in G^{\ad}(\QQ)^+$ on $X^+$ induces a map
\begin{align*}
\Gamma_{\bar{g}}\backslash X^+\rightarrow\alpha\Gamma_{\bar{g}}\alpha^{-1}\backslash X^+
\end{align*}
and this comes from a map of algebraic varieties defined over $E^{\ab}$. 

We denote the inverse limit of this system by $\underline{S}$ and denote an equivalence class in $\Gamma_{\bar{g}}\backslash X^+$ by $[\cdot]_{\Gamma_{\bar{g}}}$ so that a point in $\underline{S}$ may be identified with collection of points $[x_{\bar{g}}]_{\Gamma_{\bar{g}}}\in\Gamma_{\bar{g}}\backslash X^+$. The above action of $G^{\ad}(\QQ)^+$ on components is given by \begin{align*}
[x_{\bar{g}}]_{\Gamma_{\bar{g}}}\mapsto[\alpha x_{\bar{g}}]_{\alpha\Gamma_{\bar{g}}\alpha^{-1}}.
\end{align*} 

Now consider a point $\tilde{x}\in\underline{S}$ whose components $[x_{\bar{g}}]_{\Gamma_{\bar{g}}}$ are the images of Hodge generic points $x_{\bar{g}}\in X^+$. We would like to show that there exists a model
\begin{align*}
{\bf q}:=\langle{\bf D},{\bf S},q\rangle
\end{align*}
of $\Th^{\infty}_{\textsf{SF}}({\bf p})$ and an $x\in D$ such that, for all tuples $\bar{g}:=(e,g_1,\ldots,g_n)$,
\begin{align*}
(q(x),q(g_1x),\ldots,q(g_nx))\in Z_{\bar{g}}
\end{align*} 
is equal to the image of $[x_g]_{\Gamma_g}$ in $Z_{\bar{g}}$ under the isomorphism
\begin{align*}
\Gamma_{\bar{g}}\backslash X^+\rightarrow Z_{\bar{g}}
\end{align*}
sending $[x_{\bar{g}}]_{\Gamma_{\bar{g}}}$ to $(p(x_g),p(g_1x_g),\ldots,p(g_nx_g))$.
\begin{lem}\label{TI}
The group $\Gamma_{\infty}:=\cap_{\bar{g}} \Gamma_{\bar{g}}$ belongs to $Z_G(\QQ)$. 
\end{lem}

\begin{proof}
Note that the image of $\Gamma_{\infty}$ in $G^{\ad}(\QQ)^+$ is a normal subgroup. Therefore, it is a normal subgroup of $G^{\ad}(\RR)^+$ since $G^{\ad}(\QQ)^+$ is dense in $G^{\ad}(\RR)^+$ and $\Gamma_{\infty}$ is discrete. We deduce that the image of $\Gamma_{\infty}$ in $G^{\ad}(\QQ)^+$ is trivial and, therefore, $\Gamma_{\infty}$ is contained in $Z_G(\QQ)$.
\end{proof}
Therefore, we have an embedding of $X^+$ into $\underline{S}$ corresponding to the map
\begin{align*}
x\mapsto([x]_{\Gamma_g})_g.
\end{align*}
Denoting $\underline{p}:\underline{S}\rightarrow S(\CC)$ the natural map, we let $D$ be the union of the following two subsets of $\underline{S}$: 
\begin{enumerate}
\item[(1)] $\OOO:=\{g\tilde{x}:g\in G^{\ad}(\QQ)^+\}$
\item[(2)] $X^+\setminus\{x\in X^+:\underline{p}(x)\in\underline{p}(\OOO)\}$ 
\end{enumerate}
We then define $\bf q$ as follows: for ${\bf D}$ we take the set $D$ with its action of $G^{\ad}(\QQ)^+$, for ${\bf S}$ we take the algebraic variety $S(\CC)$ with relations for all of the Zariski closed subsets of its cartesian powers defined over $E^{\ab}(\Sigma)$ and for $q$ we take the restriction of $\underline{p}$ to $D$. Since, by Corollary \ref{com}, ${\rm T}({\bf p})$ is complete and has quantifier elimination, the following lemma demonstrates that ${\bf q}\models\Th({\bf p})$:

\begin{lem}
If $g\tilde{x}=\tilde{x}$ for some $g\in G^{\ad}(\QQ)^+$, then $g$ is the identity. 
\end{lem}

\begin{proof}
Consider the component $[x_h]_{\Gamma_h}$ of $\tilde{x}$ where, by abuse of notation, for any single element $h\in G^{\ad}(\QQ)^+$ we denote the $1$-tuple $(h)$ simply by $h$. We are supposing that 
\begin{align*}
[gx_h]_{\Gamma_{hg^{-1}}}=[x_{hg^{-1}}]_{\Gamma_{hg^{-1}}}.
\end{align*} 
Therefore,
\begin{align*}
gx_h=gh^{-1}\gamma hg^{-1}x_{hg^{-1}},
\end{align*} 
for some $\gamma\in\Gamma$. However, by compatibility, if we denote $t:=(h,hg^{-1})$, then we have a component $[x_t]_{\Gamma_t}$ such that 
\begin{align*}
 x_t=h^{-1}\gamma'hx_h=gh^{-1}\gamma''hg^{-1}x_{hg^{-1}}.
\end{align*}
It follows that $g$ belongs to the image of the intersection of the $h^{-1}\Gamma h$ for all $h\in G^{\ad}(\QQ)^+$. However, by Lemma \ref{TI}, this is trivial.
\end{proof}

\begin{lem}
The two-sorted structure ${\bf q}$ satisfies the standard fibres condition.
\end{lem}

\begin{proof}
We must show that, if $\underline{p}(g\tilde{x})=\underline{p}(\tilde{x})$, for some $g\in G^{\ad}(\QQ)^+$, then $g$ belongs to the image of $\Gamma$. The above equality implies that $gx_g=\gamma x_e$, where $x_g$ and $x_e$ are representatives for the components of $\tilde{x}$ on $g^{-1}\Gamma g\backslash X^+$ and $\Gamma\backslash X^+$, respectively, and $\gamma\in\Gamma$. Note that, by compatibility, we have another component 
\begin{align*}
[x_{(e,g)}]_{\Gamma_{(e,g)}}\in\Gamma_{(e,g)}\backslash X^+
\end{align*} 
of $\tilde{x}$ such that 
\begin{align*}
x_{(e,g)}=g^{-1}\gamma'gx_g=\gamma''x_e,
\end{align*}
where $\gamma',\gamma''\in\Gamma$. Since $x_e$ is Hodge-generic, our claim follows immediately. 
\end{proof}

Therefore, the two-sorted structure ${\bf q}$ clearly satisfies our requirements. Furthermore, the above construction immediately generalises to the case of $n$ points in $\underline{S}$ provided that their components are the images of Hodge-generic points and that they lie in distinct $G^{\ad}(\QQ)^+$ orbits.

\begin{prop}\label{numoftypes}
If the homomorphism
\begin{align*}
\Aut(\CC/L)\rightarrow\overline{\Gamma}^m
\end{align*}
associated with $\bar{z}$ does not have finite index, then the set of complete $m$-types realised over the empty set in models of $\Th^{\infty}_{\textsf{SF}}({\bf p})$ is of cardinality at least $2^{\aleph_0}$.
\end{prop}

\begin{proof}
For each $i\in\{1,...,m\}$, let $\tilde{x}_i\in\underline{S}$ be a point in the pre-image of $p(x_i)$ under the map $\underline{p}$ and let
\begin{align*}
{\bf q}:=\langle{\bf D},{\bf S},q\rangle
\end{align*}
be the corresponding model of $\Th^{\infty}_{\textsf{SF}}({\bf p})$ constructed above. 

Since $L$ is finitely generated, it suffices to consider the complete type of $(\tilde{x}_1,\ldots,\tilde{x}_m)$ over $L$. As in Proposition \ref{tp}, this is equivalent to
\begin{align}\label{uniontp}
\bigcup_{\bar{g}}{\rm qftp}_{\mathcal{L}}(q_{\bar{g}}(\tilde{x}_1),\ldots,q_{\bar{g}}(\tilde{x}_m)/L),
\end{align}
ranging over all tuples $\bar{g}=(e,g_1,\ldots,g_n)$ of distinct elements in $G^{\ad}(\QQ)^+$, where we define
\begin{align*}
q_{\bar{g}}(\tilde{x}_i):=(q(\tilde{x}_i),q(g_1\tilde{x}_i),\ldots,q(g_n\tilde{x}_i)),
\end{align*}
which is an element of $Z_{\bar{g}}$. The quantifier-free type 
\begin{align*}
{\rm qftp}_{\mathcal{L}}(q_{\bar{g}}(\tilde{x}_1),\ldots,q_{\bar{g}}(\tilde{x}_m)/L)
\end{align*}
is equivalent to the minimal algebraic subset of $Z_{\bar{g}}^m$ defined over $L$ containing the tuple. This is a subset of the fibre over $(p(x_1)),\ldots,p(x_m))$ of the finite morphism
\begin{align*}
Z_{\bar{g}}^m\rightarrow S(\CC)^m.
\end{align*}
In particular, it is $0$-dimensional. In fact, it is the $\Aut(\CC/L)$ orbit in this fibre containing the tuple. The union (\ref{uniontp}) is therefore equivalent to the $\Aut(\CC/L)$ orbit of $(\tilde{x}_1,...,\tilde{x}_m)$ in $\underline{S}$.

Since the choice of $(\tilde{x}_1,...,\tilde{x}_m)$ was arbitrary, we conclude that the number of distinct realisable complete $m$-types is at least, up to a finite factor, the index of $\Aut(\CC/L)$ in $\overline{\Gamma}^m$. However, since this index is infinite, it must be at least uncountable.

\end{proof}

\begin{proof}{\bf(Theorem \ref{OI})}
We suppose that $\Th^{\infty}_{\textsf{SF}}({\bf p})$ is $\aleph_1$-categorical. By Proposition \ref{numoftypes}, if
\begin{align*}
\Aut(\CC/L)\rightarrow\overline{\Gamma}^m
\end{align*}
is not of finite index, then the set of complete $m$-types realised over the empty set in models of $\Th^{\infty}_{\textsf{SF}}({\bf p})$ is at least uncountable. However, by Keisler's theorem this is a contradiction. Therefore, the index is finite.
\end{proof}

\subsection{Sufficient conditions}

Consider a Shimura variety $S$ of dimension $1$ and let ${\bf p}$ be the corresponding two-sorted structure. We define a pregeometry (see \cite{K}, $\mathsection$1) on the class of two-sorted structures  
\begin{align*}
{\bf q}:=\langle{\bf D},{\bf S},q\rangle\models \Th^{\infty}_{\textsf{SF}}({\bf p})
\end{align*}
by defining a closure operator $\rm cl_{\bf q}$, which on $D$ is defined as $q^{-1}\circ{\rm acl}\circ q$ and on $S(F)$ simply as $q^{-1}\circ{\rm acl}$, where $\rm acl$ denotes taking all points in $S(F)$ with coordinates in the algebraic closure of the coordinates of the points in question. 

We define the following two arithmo-geometric conditions:

\begin{con}\label{omega1}
For any finite collection $x_1,\ldots,x_m\in X^+$ of Hodge-generic points in distinct $G^{\ad}(\QQ)^+$ orbits, the induced homomorphism
\begin{align*}
\Aut(\CC/L)\rightarrow\overline{\Gamma}^m,
\end{align*}
where we denote by $L$ the field generated over $E^{\ab}(\Sigma)$ by the coordinates of the $p(x_i)$, has finite index.
\end{con}

\begin{con}\label{omega2}
For any algebraically closed field $F\subset\CC$ and any finite collection $x_1,\ldots,x_m\in X^+$ of Hodge-generic points in distinct $G^{\ad}(\QQ)^+$ orbits such that 
\begin{enumerate}
\item[(1)] $p(x_i)\notin S(F)$ for all $i\in\{1,\ldots,m\}$ and
\item[(2)] the field $L$ generated over $F$ by the coordinates of the $p(x_i)$ has transcendence degree $1$,
\end{enumerate}
the induced homomorphism
\begin{align*}
\Aut(\CC/L)\rightarrow\overline{\Gamma}^m,
\end{align*}
has finite index.
\end{con}

We recall again that though these indices are well-defined, the homomorphism itself is only well-defined up to conjugation.

\begin{teo}\label{AK}
If Conditions \ref{omega1} and \ref{omega2} are satisfied, $\Th^{\infty}_{\textsf{SF}}({\bf p})$ is $\kappa$-categorical for all infinite cardinalities $\kappa$.
\end{teo}

First we will prove the following:

\begin{lem}\label{A1}
If Conditions \ref{omega1} and \ref{omega2} are satisfied, $\Th^{\infty}_{\textsf{SF}}({\bf p})$ is $\kappa$-categorical for $\kappa\leq\aleph_1$.
\end{lem}


We refer to \cite{K}: note that the pregeometry axioms I.1, I.2, I.3  and the $\aleph_0$-homogeneity axiom II.1 of \cite{K}, \S1 are easily verified. Thus, by \cite{K}, Corollary 2.2, $\kappa$-categoricity for $\kappa\leq\aleph_1$ relies on axiom II.2.

\begin{lem}\label{L1}
If Condition \ref{omega1} holds then the class of models of $\Th^{\infty}_{\textsf{SF}}({\bf p})$ satisfies $\aleph_0$-homogeneity over the empty set. 
\end{lem}

\begin{proof}
Suppose that
\begin{align*}
{\bf q}:=\langle{\bf D},{\bf S},q\rangle,\ {\bf q'}:=\langle{\bf D'},{\bf S'},q'\rangle
\end{align*}
are two models of $\Th^{\infty}_{\textsf{SF}}({\bf p})$ and suppose that $x_1,\ldots,x_m\in D$ is a collection of Hodge-generic points in distinct $G^{\ad}(\QQ)^+$ orbits. Note that we may ignore the possibility of any of the $x_i$ belonging to $S(F)$. Suppose that $x'_1,\ldots,x'_m\in D'$ is another collection of Hodge-generic points in distinct $G^{\ad}(\QQ)^+$ orbits such that
\begin{align*}
{\rm qftp}_{\mathcal{L}}(x_1,\ldots,x_m)={\rm qftp}_{\mathcal{L}}(x'_1,\ldots,x'_m).
\end{align*}
Equivalently, we have a finite partial isomorphism between ${\bf q}$ and ${\bf q'}$. If we denote by $L$ the field generated over $E^{\ab}(\Sigma)$ by the coordinates of the $q(x_i)$, then this isomorphism on the variety structures yields an embedding $\sigma$ of $L$ into $F'$ fixing $E^{\ab}(\Sigma)$.

Let $\tau\in D$ be a Hodge-generic point in a separate $G^{\ad}(\QQ)$ orbit from the $x_i$ such that the coordinates of $q(\tau)$ belong to the algebraic closure $\bar{L}$ of $L$ in $F$. We must show that there exists a $\tau'\in D'$ such that
\begin{align*}
{\rm qftp}_{\mathcal{L}}(x_1,\ldots,x_m,\tau)={\rm qftp}_{\mathcal{L}}(x'_1,\ldots,x'_m,\tau').
\end{align*}
This suffices to prove the lemma since, if $\tau$ were a point in $S(F)$, we would proceed as in the proof of Proposition \ref{qec}. By Proposition \ref{tp}, the left-hand side is determined by
\begin{align}\label{1}
\bigcup_{\bar{g}}{\rm qftp}_{\mathcal{L}}(q_{\bar{g}}(x_1),\ldots,q_{\bar{g}}(x_m),q_{\bar{g}}(\tau))
\end{align}
over all tuples ${\bar{g}}:=(e,g_1,\ldots,g_n)$. For such a tuple we denote by $L_{\bar{g}}$ the field generated over $L$ by the coordinates of the $q_{\bar{g}}(x_i)$ and $q_{\bar{g}}(\tau)$ and we let
\begin{align*}
\overline{\Gamma}_{\bar{g}}:=\varprojlim_{{\bar{h}}}\Gamma_{\bar{g}}/\Gamma_{{\bar{h}}},
\end{align*}	
where ${\bar{h}}$ varies over all tuples such that $\Gamma_{{\bar{h}}}$ is contained in $\Gamma_{\bar{g}}$ and is normal in $\Gamma$. 

Choose an embedding of $\bar{L}$ in $\CC$ yielding, in particular, an embedding of $S(L_{(e)})$ in $S(\CC)$. By Condition \ref{omega1}, the homomorphism 
\begin{align*}
\Aut(\CC/L_{(e)})\rightarrow\overline{\Gamma}^{m+1}
\end{align*}
corresponding to the image of $(q(x_1),\ldots,q(x_m),q(\tau))$ in $S(\CC)^{m+1}$ has finite index. Therefore, there exists a tuple ${\bar{g}}$ such that the homomorphism
\begin{align*}
\Aut(\CC/L_{\bar{g}})\rightarrow\overline{\Gamma}_{\bar{g}}^{m+1}
\end{align*}
corresponding to the image of $(q_{\bar{g}}(x_1),\ldots,q_{\bar{g}}(x_m),q_{\bar{g}}(\tau))$ in $Z_{\bar{g}}^{m+1}$ is surjective. 

It is a simple property of types that (\ref{1}) is equivalent to
\begin{align*}
{\rm qftp}_{\mathcal{L}}(q_{\bar{g}}(x_1),\ldots,q_{\bar{g}}(x_m),q_{\bar{g}}(\tau))\cup\bigcup_{\bar{h}}{\rm qftp}_{\mathcal{L}}(q_{\bar{h}}(x_1),\ldots,q_{\bar{h}}(x_m),q_{\bar{h}}(\tau)/L_{\bar{g}}),
\end{align*}
again varying over all tuples ${\bar{h}}$ such that $\Gamma_{{\bar{h}}}$ is contained in $\Gamma_{\bar{g}}$ and is normal in $\Gamma$. Therefore, we choose a $\tau'\in D'$ such that
\begin{align*}
{\rm qftp}_{\mathcal{L}}(q'_{\bar{g}}(x'_1),\ldots,q'_{\bar{g}}(x'_m),q'_{\bar{g}}(\tau'))={\rm qftp}_{\mathcal{L}}(q_{\bar{g}}(x_1),\ldots,q_{\bar{g}}(x_m),q_{\bar{g}}(\tau)),
\end{align*}
the existence of which is guaranteed by Proposition \ref{real}. This induces an extension of $\sigma$ to $L_{\bar{g}}$ and we claim that 
\begin{align}\label{2}
{\rm qftp}_{\mathcal{L}}(q'_{\bar{h}}(x'_1),\ldots,q'_{\bar{h}}(x'_m),q'_{\bar{h}}(\tau')/\sigma(L_{\bar{g}}))
\end{align}
is equal to
\begin{align}\label{3}
{\rm qftp}_{\mathcal{L}}(q_{\bar{h}}(x_1),\ldots,q_{\bar{h}}(x_m),q_{\bar{h}}(\tau)/L_{\bar{g}}).
\end{align}
To see this, recall that we have a finite morphism $\pi: Z^{m+1}_{\bar{h}}\rightarrow Z^{m+1}_{\bar{g}}$ defined over $E^{\ab}$ and (\ref{3}) is determined by the $\Aut(\CC/L_{\bar{g}})$ orbit containing 
\begin{align*}
(q_{\bar{h}}(x_1),\ldots,q_{\bar{h}}(x_m),q_{\bar{h}}(\tau))\in Z_{\bar{h}}^{m+1}
\end{align*} 
in the fibre over 
\begin{align*}
(q_{\bar{g}}(x_1),\ldots,q_{\bar{g}}(x_m),q_{\bar{g}}(\tau))\in Z_{\bar{g}}^{m+1}.
\end{align*}
However, this orbit is the entire fibre $\mathcal{F}$ since the induced homomorphism
\begin{align*}
\varphi:\Aut(\CC/L_{\bar{g}})\rightarrow(\Gamma_{\bar{g}}/\Gamma_{{\bar{h}}})^{m+1}
\end{align*}
is surjective. Similarly, (\ref{2}) is determined by the $\Aut(\CC/\sigma(L_{\bar{g}}))$ orbit containing 
\begin{align*}
(q'_{\bar{h}}(x'_1),\ldots,q'_{\bar{h}}(x'_m),q'_{\bar{h}}(\tau'))\in Z_{\bar{h}}^{m+1}
\end{align*} 
in the fibre over 
\begin{align*}
(q'_{\bar{g}}(x'_1),\ldots,q'_{\bar{g}}(x'_m),q'_{\bar{g}}(\tau'))\in Z_{\bar{g}}^{m+1}.
\end{align*}
Our claim will follow, hence completing the proof, if we can show that this orbit is also the entire fibre and equal to $\sigma(\mathcal{F})$. To see this, choose an extension of $\sigma$ to $\Aut(\CC/L_{\bar{g}})$. Let
\begin{align*}
y'\in\pi^{-1}(q'_{\bar{g}}(x'_1),\ldots,q'_{\bar{g}}(x'_m),q'_{\bar{g}}(\tau')),
\end{align*} 
equal to $\sigma(y)$, where
\begin{align*}
y\in\pi^{-1}((q_{\bar{g}}(x_1),\ldots,q_{\bar{g}}(x_m),q_{\bar{g}}(\tau))
\end{align*}
and let $\tau\in\Aut(\CC/\sigma(L_{\bar{g}}))$. Therefore, $\sigma^{-1}\tau\sigma\in\Aut(\CC/L_{\bar{g}})$ and
\begin{align*}
\tau(y')=\tau\sigma(y)=\sigma\sigma^{-1}\tau\sigma(y)=\sigma(\varphi(\sigma^{-1}\tau\sigma)(y))=\varphi(\sigma^{-1}\tau\sigma)(y').
\end{align*}
Therefore, the homomorphism
\begin{align*}
\Aut(\CC/\sigma(L_{\bar{g}}))\rightarrow(\Gamma_{\bar{g}}/\Gamma_{{\bar{h}}})^{m+1}
\end{align*}
sends $\tau$ to $\phi(\sigma^{-1}\tau\sigma)$. In particular, it is surjective and the claim follows.
\end{proof}

\begin{lem}\label{L2}
If Condition \ref{omega2} holds then the class of models of $\Th^{\infty}_{\textsf{SF}}({\bf p})$ satisfies $\aleph_0$-homogeneity over countable models. 
\end{lem}

\begin{proof}
The argument is identical to the proof Theorem \ref{L1}, this time working over the closed subset and applying Condition \ref{omega2}. 
\end{proof}

\begin{proof}{\bf(Lemma \ref{A1})}
By Lemmas \ref{L1} and \ref{L2}, this is immediate from \cite{K}, Corollary 2.2.
\end{proof}

\begin{rem}
Suppose we wanted to consider the case of $F$ having finite transcendence degree $n$. We could replace the $\Th^{\infty}_{\textsf{SF}}({\bf p})$-axiom of infinite transcendence degree with the axiom of having transcendence degree $n$. By \cite{K}, Corollary 2.2, Conditions \ref{omega1} and \ref{omega2} imply that this theory is $\aleph_0$-categorical.
\end{rem}

\begin{proof}{\bf (Theorem \ref{AK})}
Under our assumptions, the model ${\bf p}$ with the pregeometry $\rm cl_p$ is a quasiminimal pregeometry structure and the class $\mathcal{K}({\bf p})$ defined in \cite{BHHKK} is quasiminimal. So, by the
main result Theorem 2.2 of loc. cit., $\mathcal{K}({\bf p})$ has a unique model (up to isomorphism)
in each infinite cardinality. In particular, $\mathcal{K}({\bf p})$ contains a unique model ${\bf p_0}$ of
cardinality $\aleph_0$. 

Let $\mathcal{K}$ be the class of models of $\Th^{\infty}_{\textsf{SF}}({\bf p})$. It is clear
that $\mathcal{K}({\bf p})\subset\mathcal{K}$ since ${\bf p}\in\mathcal{K}$. Thus, to prove the theorem we need to show that $\mathcal{K}\subset\mathcal{K}({\bf p})$. 

By Lemma \ref{A1}, $\mathcal{K}$ has a unique model ${\bf q_0}$ of cardinality $\aleph_0$. As a class of models of an $\mathcal{L}_{\omega_1,\omega}$-sentence, $\mathcal{K}$ is also an abstract elementary class with
Lowenheim-Skolem number $\aleph_0$. So, by the downward Lowenheim-Skolem theorem, every model
in $\mathcal{K}$ is a direct limit of copies of the unique
model of cardinality $\aleph_0$ (with elementary embeddings as morphisms). Finally, all the embeddings in $\mathcal{K}$ are closed with respect to the pregeometry. Therefore, 
\begin{align*}
\mathcal{K}=\mathcal{K}({\bf q_0})=\mathcal{K}({\bf p_0})=\mathcal{K}({\bf p}).
\end{align*}
\end{proof}

\section{Arithmetic}

Finally, we conclude the proof of Theorem \ref{main} by proving the following model-theoretic result:

\begin{teo}\label{mainmodel}
Let $S$ be a modular or Shimura curve and let $\bf p$ be the corresponding two-sorted structure. Then $\Th^{\infty}_{\textsf{SF}}({\bf p})$ is $\kappa$-categorical for all infinite cardinalities $\kappa$.
\end{teo}

By Theorem \ref{AK}, we must verify the Conditions \ref{omega1} and \ref{omega2} in the case that $S$ is either a modular or Shimura curve. We treat first the case of a modular curve.

\subsection{Conditions \ref{omega1} and \ref{omega2} for modular curves}

Recall that $G$ is $\GL_2$ and $X^+$ is the upper half-plane $\HH$. The reflex field is $\QQ$ whose maximal abelian extension $\QQ^{\ab}$ is obtained by adjoining all roots of unity. Since $K$ is contained in $\GL_2(\hat{\ZZ})$, $\Gamma$ is contained in $\SL_2(\ZZ)$.

Consider the quotients $\Gamma/\Gamma(N)$ for those $\Gamma(N)$ contained in $\Gamma$. For any tuple $\bar{g}:=(e,g_1,\ldots,g_n)$ of elements belonging to $\GL_2(\QQ)^+$ such that $\Gamma_{\bar{g}}$ is a normal subgroup of $\Gamma$, and an $N$ such that $\Gamma(N)$ is contained in $\Gamma_{\bar{g}}$, we have a short exact sequence
\begin{align*}
\{1\}\rightarrow\Gamma_{\bar{g}}/\Gamma(N)\rightarrow\Gamma/\Gamma(N)\rightarrow\Gamma/\Gamma_{\bar{g}}\rightarrow\{1\}.
\end{align*}
If we let 
\begin{align*}
\hat{\Gamma}:=\varprojlim_N\Gamma/\Gamma(N)
\end{align*}
then, since the groups $\Gamma_{\bar{g}}/\Gamma(N)$ are finite, the induced map $\hat{\Gamma}\rightarrow\overline{\Gamma}$ is surjective, by the Mittag-Leffler condition. 

Let $\tau_1,\ldots,\tau_m\in\HH$ be a collection of Hodge-generic points in distinct $\GL_2(\QQ)^+$ orbits, inducing a conjugacy class of homomorphisms
\begin{align*}
\Aut(\CC/L)\rightarrow\overline{\Gamma}^m
\end{align*}
where we denote by $L$ the field generated over $\QQ^{\ab}(\Sigma)$ by the coordinates of the $p(\tau_i)$. In the same way, the tower of modular curves associated to the $\Gamma(N)$ gives rise to another conjugacy class of homomorphisms
\begin{align}\label{4}
\Aut(\CC/L)\rightarrow\hat{\Gamma}^m
\end{align}
and any two homomorphisms corresponding to compatible points in the fibre commute with the surjection $\hat{\Gamma}\rightarrow\overline{\Gamma}$. Therefore, if we show that (\ref{4}) has finite index, then Condition \ref{omega1} follows.

Note that $\hat{\Gamma}$ is a finite index subgroup of $\SL_2(\hat{\ZZ})$ and it is well known that we can associate with $\tau_i$ an elliptic curve $E_i$ defined over $\QQ(j_i)$, where $j_i:=j(E_i)=j(\tau_i)$, such that (\ref{4}) is the restriction of
\begin{align}\label{5}
\Aut(\CC/\QQ^{\ab}(j_1,\ldots,j_m))\rightarrow\SL_2(\hat{\ZZ})^m
\end{align}
coming from the Galois representation on the Tate module of the product of the $E_i$. In the case that the $j_i$ are algebraic, then the image of (\ref{5}) is open by \cite{R}, Theorem 3.5. Note that we are using a theorem of Faltings here that, because the $E_i$ are mutually non-isogenous (since the $\tau_i$ belong to distinct $\GL_2(\QQ)^+$ orbits), the systems of $p$-adic representations associated to any $E_i$ and $E_j$ do not become isomorphic over a finite extension. One  may reduce to the case that the $j_i$ are algebraic by the usual specialisation arguments (see, for example, \cite{P}, Remark 6.12). 

We recall that $\SL_2(\hat{\ZZ})^m$ satisfies the condition of \cite{R}, Lemma 3.4. If $U$ denotes the open image of (\ref{5}) then it necessarily contains an open subgroup $U'$ that is a product of open subgroups in the $m$ factors. By the aforementioned condition, the closure of the commutator subgroup of $U'$ is open in $U'$. It is also contained in the commutator subgroup of $U$ and so the closure of the commutator subgroup of $U$ is open in $U$.

If $\QQ^{\ab}(\Sigma)(j_1,\ldots,j_m)$ were an abelian extension of $\QQ^{\ab}(j_1,\ldots,j_m)$, then the image of $\Aut(\CC/\QQ^{\ab}(\Sigma)(j_1,\ldots,j_m))$ would necessarily contain the commutator subgroup of $U$. However, it is also closed as a compact subset of a Hausdorff topological space. Therefore, it would be open and, therefore, of finite index. To conclude the proof we then restrict to $\Aut(\CC/L)$, which is clearly a finite index subgroup. Therefore, Condition \ref{omega1} follows from the following lemma:

\begin{lem}
$\QQ^{\ab}(\Sigma)(j_1,\ldots,j_m)$ is an abelian extension of $\QQ^{\ab}(j_1,\ldots,j_m)$.
\end{lem}

\begin{proof}
It clearly suffices to show that $\QQ^{\ab}(\Sigma)$ is an abelian extension of $\QQ^{\ab}$. Note that $\Sigma$ is the set of $j(\tau)$ for all special points $\tau\in\HH$. Any such $\tau$ generates an imaginary quadratic field $\QQ(\tau)$ and $\QQ(\tau,j(\tau))$ is the Hilbert class field of $\QQ(\tau)$. In particular, it is abelian. On the other hand, since $\QQ(\tau)$ is quadratic, it is contained in $\QQ^{\ab}$. Therefore, $\QQ^{\ab}(j(\tau))$ is an abelian extension of $\QQ^{\ab}$. Since the compositum of abelian extensions is abelian, the result follows.
\end{proof}

Therefore, it remains to verify Condition \ref{omega2}. To that end, we let $F$ denote an algebraically closed subfield of $\CC$ and assume that $p(\tau_i)\notin S(F)$ for all $i\in\{1,\ldots,m\}$. Furthermore, we let $L$ denote the field generated over $F$ by the coordinates of the $p(\tau_i)$ and assume that the transcendence degree of $L$ over $F$ is $1$. We return to the homomorphism (\ref{5}) and make the following deductions:

\begin{enumerate}

\item[(a)] $F(j_1,\ldots,j_m)$ is a normal extension of $\QQ(j_1,\ldots,j_m)$ since it is isomorphic to the extension $F$ over $/F\cap\QQ(j_1,\ldots,j_m)$ and $F$ is a normal extension $\QQ$. Therefore, $\Aut(\CC/F(j_1,\ldots,j_m))$ is a normal subgroup of $\Aut(\CC/\QQ^{\ab}(j_1,\ldots,j_m))$.

\item[(b)] The image of the projection of the homomorphism
\begin{align*}
\Aut(\CC/F(j_1,\ldots,j_m))\rightarrow\SL_2(\hat{\ZZ})^m
\end{align*}
to any of the individual factors has finite index. This is simply because $\Gamma(N)\backslash\HH$ and the projection to $\Gamma(1)\backslash\HH$ are both defined over $F$, as are the automorphisms induced by the elements of $\SL_2(\ZZ/N\ZZ)$. Furthermore, $F(j_1,\ldots,j_m)$ is at most a finite extension of $F(j_i)$, the function field of the base. The image is closed because it is a compact subset of a Hausdorff topological space. Therefore, it is open.

\item[(c)] We claim that the projection of the homomorphism
\begin{align*}
\Aut(\CC/F(j_1,\ldots,j_m))\rightarrow\SL_2(\FF_p)^m
\end{align*}
to the product of any two factors is surjective for almost all primes $p$. To see this note that the projection to any single factor is surjective for almost all primes by (b). Therefore, let
\begin{align*}
\Aut(\CC/F(j_1,\ldots,j_m))\rightarrow\SL_2(\FF_p)^2
\end{align*}
be a projection as in the claim. For almost all primes $p$, Goursat's Lemma states that the image is either full or its image in
\begin{align*}
\SL_2(\FF_p)/\ker(\pi_2)\times\SL_2(\FF_p)/\ker(\pi_1)
\end{align*}
is the graph of an isomorphism
\begin{align*}
\SL_2(\FF_p)/\ker(\pi_2)\cong\SL_2(\FF_p)/\ker(\pi_1),
\end{align*}
where $\pi_1$ and $\pi_2$ denote the first and second projections, respectively. However, if $p>3$, then $\SL_2(\FF_p)/\{\pm e\}$ is simple and so $\ker(\pi_1)$ and $\ker(\pi_2)$ are equal to $\SL_2(\FF_p)$, $\{\pm e\}$ or $\{e\}$. The latter two possibilities can occur for only finitely many primes since the projection of the homomorphism
\begin{align*}
\Aut(\CC/\QQ^{\ab}(j_1,\ldots,j_m)\rightarrow\SL_2(\FF_p)^m,
\end{align*}
to the product of any two factors is surjective for almost all primes and $\Aut(\CC/F(j_1,\ldots,j_m))$ is normal in $\Aut(\CC/\QQ^{\ab}(j_1,\ldots,j_m)$ by (a).

\item[(d)] The projection of the homomorphism
\begin{align*}
\Aut(\CC/F(j_1,\ldots,j_m))\rightarrow\SL_2(\ZZ_p)^m,
\end{align*}
to the product of any two factors is surjective for almost all primes $p$ by \cite{S}, \S6, Lemma 10 and (c).

\item[(e)] We claim that the projection of the homomorphism
\begin{align*}
\Aut(\CC/F(j_1,\ldots,j_m))\rightarrow\SL_2(\ZZ_p)^m,
\end{align*}
to the product of any two factors has open image for all primes $p$. Indeed, for any such projection, let $\mathfrak{a}_p$ denote the $\QQ_p$-lie algebra of the image of $\Aut(\CC/F(j_1,\ldots,j_m))$ in 
\begin{align*}
\SL_2(\ZZ_p)\times\SL_2(\ZZ_p).
\end{align*} 
By (b), it is a $\QQ_p$-lie subalgebra of $\mathfrak{sl}_2\times\mathfrak{sl}_2$ that surjects on to each factor. Therefore, as in \cite{S}, \S6, Lemma 7, either $\mathfrak{a}_p=\mathfrak{sl}_2\times\mathfrak{sl}_2$ or it is the graph of a $\QQ_p$-algebra isomorphism $\mathfrak{sl}_2\cong\mathfrak{sl}_2$. However, the latter is impossible using fact (a) that $\Aut(\CC/F(j_1,\ldots,j_m))$ is a normal subgroup of $\Aut(\CC/\QQ^{\ab}(j_1,\ldots,j_m)$.   

\end{enumerate}

Therefore, the image of the projection of the homomorphism
\begin{align*}
\Aut(\CC/L)\rightarrow\SL_2(\hat{\ZZ})^m
\end{align*}
to the product of any two factors is open by (d) and (e). Therefore, the image of this homomorphism itself is open by \cite{R}, Lemma 3.4 and Remark 3. Therefore, it is of finite index. 

\subsection{Conditions \ref{omega1} and \ref{omega2} for Shimura curves}

Now we deal with Shimura curves. Recall that $G$ is $\Res_{B/\QQ}\GG_{m,B}$, where $B$ is a quaternion division algebra over $\QQ$ split at infinity and $X^+$ is again the upper half-plane $\HH$. The reflex field is also $\QQ$. 

We can choose a basis of $B$ such that, under the induced embedding of $G$ in $\GL_4$, $K$ is contained in $\GL_4(\hat{\ZZ})$. Given this representation, we can define the principal congruence subgroups $\Gamma(N)$ of $G$. We can also define the congruence subgroups of $G$ in the usual way as those subgroups of $G(\QQ)$ containing some $\Gamma(N)$ as a subgroup of finite index. The latter notion does not depend on the representation, of course, and $\Gamma$ is always a congruence subgroup. We define $G':=G\cap\SL_4$, which is the group $\Res_{B/\QQ}\GG_{m,B}$ of norm $1$ elements in $B$.

Let $\tau_1,\ldots,\tau_m\in\HH$ be a collection of Hodge-generic points in distinct $G(\QQ)_+$ orbits. Let $L$ denote the field generated over $\QQ^{\ab}(\Sigma)$ by the coordinates of the $p_i:=p(\tau_i)$. By abuse of notation, for any field $F$, we will denote by $F(p_1,\ldots,p_m)$ the field obtained from $F$ by adjoining the coordinates of the $p_i$. Exactly as before, we may pass to the (conjugacy class of a) homomorphism
\begin{align}\label{6}
\Aut(\CC/L)\rightarrow\hat{\Gamma}^m,
\end{align}
noting that $\hat{\Gamma}$ is an open subgroup of $G'(\AAA_f)$.

Again one can associate with $\tau_i$ an abelian surface $A_i$ with quaternionic multiplication defined over $\QQ(p_i)$ such that (\ref{6}) is the restriction of
\begin{align}\label{7}
\Aut(\CC/\QQ^{\ab}(p_1,\ldots,p_m))\rightarrow G'(\AAA_f)^m
\end{align}
coming from the Galois represention on the Tate module of the product of the $A_i$. In the case that the $p_i$ are algebraic and $m=1$, Ohta proved in \cite{O} that the image of (\ref{7}) was open. In the case that the $p_i$ are not algebraic, one can use specialisation arguments as before.

A recent result of Cadoret (see \cite{C}, Theorem 1.2) states that, in general, the image of (\ref{7}) is open if and only if it is open at some prime $p$. Therefore, we choose a prime $p$ such that $G'_{\QQ_p}$ is isomorphic to $\SL_{2,\QQ_p}$. Then the fact that the projection of
\begin{align}\label{8}
\Aut(\CC/\QQ^{\ab}(p_1,\ldots,p_m))\rightarrow\SL_2(\QQ_p)^m
\end{align}
to the product of any two factors has open image can be proved using Ohta's result, precisely as in \cite{S}, Lemme 7. By \cite{R}, Lemma 3.4 and Remark 3, the image of (\ref{8}) itself is open. Therefore, the result of Cadoret implies that the image of (\ref{7}) is open.

Since $G'_{\QQ_p}$ is semisimple, and isomorphic to $\SL_{2,\QQ_p}$ for all but finitely many primes, we can conclude that the image of $\Aut(\CC/\QQ^{\ab}(\Sigma)(p_1,\ldots,p_m))$ in $G'(\AAA_f)^m$ is open by \cite{R}, Remark 2 and Remark 3, precisely as in the case of modular curves given the following lemma:

\begin{lem}
$\QQ^{\ab}(\Sigma)(p_1,\ldots,p_m)$ is an abelian extension of $\QQ^{\ab}(p_1,\ldots,p_m)$.
\end{lem}

\begin{proof}
Again, it suffices to show that $\QQ^{\ab}(\Sigma)$ is an abelian extension of $\QQ^{\ab}$. A special point $\tau$ corresponds to an abelian surface that is isogenous to the square of an elliptic curve with complex multiplication by an imaginary quadratic field $k$. By a theorem of Shimura, $k(p(\tau))$ is the Hilbert class field of $k$. Again, $k$ is contained in $\QQ^{\ab}$ so $\QQ^{\ab}(p(\tau))$ is an abelian extension of $\QQ^{\ab}$. The claim follows after taking the compositum over all special points. 
\end{proof}

Therefore, it remains to verify Condition \ref{omega2}. Again, we let $F$ denote an algebraically closed subfield of $\CC$ and assume that $p_i\notin S(F)$ for all $i\in\{1,\ldots,m\}$. Furthermore, we let $L$ denote the field $F(p_1,\ldots,p_m)$ and assume that the transcendence degree of $L$ over $F$ is $1$. We return to the homomorphism (\ref{7}) and make the following deductions:

\begin{enumerate}

\item[(a)] As before, $\Aut(\CC/L)$ is a normal subgroup of $\Aut(\CC/\QQ^{\ab}(p_1,\ldots,p_m))$.

\item[(b)] Again, the image of the projection of the homomorphism
\begin{align*}
\Aut(\CC/L)\rightarrow G'(\AAA_f)^m
\end{align*}
to any of the individual factors is open.

\item[(c)] For almost all primes $p$ such that $G'_{\QQ_p}$ is isomorphic to $\SL_{2,\QQ_p}$, the projection of the homomorphism
\begin{align*}
\Aut(\CC/L))\rightarrow\SL_2(\FF_p)^m
\end{align*}
to the product of any two factors is surjective. This is precisely as before.

\item[(d)] Again, for almost all primes $p$ such that $G'_{\QQ_p}$ is isomorphic to $\SL_{2,\QQ_p}$, the projection of the homomorphism
\begin{align*}
\Aut(\CC/L)\rightarrow\SL_2(\ZZ_p)^m,
\end{align*}
to the product of any two factors is surjective.

\item[(e)] We claim that the projection of the homomorphism
\begin{align*}
\Aut(\CC/L)\rightarrow G'(\QQ_p)^m,
\end{align*}
to the product of any two factors has open image for all primes $p$. Indeed let $\mathfrak{g}$ denote the lie algebra of $G'_{\QQ_p}$ and, for any projection as above, let $\mathfrak{a}_p$ denote the $\QQ_p$-lie algebra of the image of $\Aut(\CC/L)$ in 
\begin{align*}
G'(\QQ_p)\times G'(\QQ_p).
\end{align*} 
By (b), it is a $\QQ_p$-lie subalgebra of $\mathfrak{g}\times\mathfrak{g}$ that surjects on to each factor. Therefore, by Goursat's Lemma for algebras, either $\mathfrak{a}_p=\mathfrak{g}\times\mathfrak{g}$ or the projections to either factor are proper two-sided ideals $\mathfrak{h}_1$ and $\mathfrak{h_2}$ of $\mathfrak{g}$ and the image of $\mathfrak{a}_p$ in $\mathfrak{g}/\mathfrak{h}_1\times\mathfrak{g}/\mathfrak{h}_2$ is the graph of a $\QQ_p$-algebra isomorphism $\mathfrak{g}/\mathfrak{h}_1\cong\mathfrak{g}/\mathfrak{h}_2$. However, the latter is impossible by (a) since the lie algebra of the image of $\Aut(\CC/\QQ^{\ab}(p_1,\ldots,p_m))$ is $\mathfrak{g}\times\mathfrak{g}$.

\end{enumerate} 

Therefore, the image of the projection of the homomorphism
\begin{align*}
\Aut(\CC/L)\rightarrow G'(\AAA_f)^m
\end{align*}
to the product of any two factors is open by (d) and (e). Therefore, the image of this homomorphism itself is open by \cite{R}, Lemma 3.4, Remark 2 and Remark 3. Therefore, (\ref{6}) has finite index.

\subsection{Reduction to the torsion free case}\label{red}

Suppose that $(G,X)$ is a Shimura datum and that Theorem \ref{main} holds for any Shimura variety defined by a congruence subgroup that is torsion free in $G^{\ad}(\QQ)^+$. Consider an arbitrary congruence subgroup $\Gamma$ contained in $G(\QQ)_+$ that, nonetheless, gives rise to a locally symmetric variety $\Gamma\backslash X^+$ with a canonical model $S$ defined over $E^{\ab}$. We denote by
\begin{align*}
p:X^+\rightarrow S(\CC)
\end{align*}
the usual map and consider another map $q$ satisfying the standard fibres, special points and modularity conditions.

One can choose finitely many elements $g_1,\ldots,g_n\in G(\QQ)_+$ such that
\begin{align*}
\Gamma':=\Gamma\cap g^{-1}\Gamma g_1\cap\cdots\cap g^{-1}_n\Gamma g_n
\end{align*}
is a congruence subgroup whose image in $G^{\ad}(\QQ)^+$ is torsion free. Therefore, $\Gamma'\backslash X^+$ is a locally symmetric variety possessing a canonical model $S'$ defined over $E^{\ab}$ and the natural map
\begin{align*}
\Gamma'\backslash X^+\rightarrow\Gamma\backslash X^+
\end{align*}
comes from an algebraic map $\pi:S'\rightarrow S$ defined over $E^{\ab}$. 

We define
\begin{align*}
p':X^+\rightarrow S'(\CC):x\mapsto (p(x),p(g_1x),\ldots,p(g_nx))
\end{align*}
and similarly
\begin{align*}
q':X^+\rightarrow S'(\CC):x\mapsto (q(x),q(g_1x),\ldots,q(g_nx)),
\end{align*}
which automatically satisfies the standard fibres, special points and modularity conditions. We also note that $p'$ is the usual map associated with $S'$, $p=\pi\circ p'$ and $q=\pi\circ q'$. Applying Theorem \ref{main}, we obtain a $G^{\ad}(\QQ)^+$-equivariant bijection $\varphi$ and an automorphism $\sigma$ of $\CC$ fixing $E^{\ab}(\Sigma)$ such that
\begin{center}
$\begin{CD}
X^+ @>\varphi>> X^+\\
@VVp' V @VVq' V\\
S(\CC) @>\sigma>> S(\CC)
\end{CD}$
\end{center}
commutes. Since $\pi$ is defined over $E^{\ab}$, we obtain a commutative diagram
\begin{center}
$\begin{CD}
X^+ @>\varphi>> X^+\\
@VVp V @VVq V\\
S(\CC) @>\sigma>> S(\CC).
\end{CD}$
\end{center}

\section*{Acknowledgements}
The authors would like to thank Martin Bays, Daniel Bertrand, Minhyong Kim, Angus Macintyre, Jonathan Pila, Andrei Yafaev and Boris Zilber for many interesting discussions during the preparation of this article. Special thanks from both authors must go to Misha Gavrilovich for invitations to the Kurt G\"odel Research Center, Vienna, and the first author would also like to thank Misha Gavrilovich for introducing him to the problem. The first author would like to thank Jonathan Kirby and the EPSRC (grant number EP/L006375/1) for the opportunity to visit the University of East Anglia. He would also like to thank the Institut des Hautes Etudes Scientifiques and University College London. The second author would like to thank the University of East Anglia and the University of Oxford. Finally, both authors would like to heartily thank the referee for their many insightful comments.

\end{document}